\documentclass{amsart}
\usepackage[utf8]{inputenc}
\usepackage{a4wide}
\usepackage{euscript}
\usepackage{fullpage}
\usepackage[T2A]{fontenc}
\usepackage{amsfonts}
\usepackage{amssymb, amsthm}
\usepackage{amsmath}
\usepackage{mathtools}
\usepackage{graphicx}
\usepackage{geometry}
\usepackage{tikz}
\usepackage{bbm}
\usepackage{mathabx}

\numberwithin{equation}{section}

\theoremstyle{plain}
\newtheorem{thm}{Тheorem}

\newtheorem{lemma}{Lemma}
\newtheorem{st}{Statement}

\theoremstyle{definition}

\theoremstyle{remark}

\newcommand{\R}{\mathbb{R}}
\newcommand{\B}{\mathbb{B}}
\newcommand{\dd}{\textrm{d}}
\newcommand{\PP}{\mathbb{P}}
\newcommand{\E}{\mathbb{E}}
\newcommand{\conv}{\textrm{conv}}
\newcommand{\Vol}{\textrm{Vol}}
\newcommand{\Pol}[3]{\mathcal{P}_{#1,#2}^{\widebar{#3}}}

\numberwithin{thm}{section}
\numberwithin{lemma}{section}
\numberwithin{st}{section}
\numberwithin{cor}{section}

\title{Mixed random beta-polytopes}


\author[T.~Moseeva]{Tatiana Moseeva}
\address{Tatiana Moseeva, Saint Petersburg University, 7/9 Universitetskaya nab., St. Petersburg, 199034 Russia
}
\email{polezina@yandex.ru}

\begin{document}

\thanks{This work was performed at the Saint Petersburg Leonhard Euler International Mathematical Institute and supported by the Ministry of Science and Higher Education of the Russian Federation (agreement no. 075–15–2022–287)
and  by the Foundation for the
Advancement of Theoretical Physics and Mathematics ``BASIS''.}

\begin{abstract}
In this paper, we generalize the result on the average volume of random polytopes with vertices following beta distributionsto the case of non-identically distributed vectors. Specifically,we consider the convex hull of independent random vectors in $\mathbb{R}^d$, where each vector follows a beta distribution with potentially different parameters. We derive an expression for the expected volume of these generalized beta--polytopes. Additionally, we compute the expected value of a functional introduced by Wieacker, which involves the distance of facets from the origin and their volumes.Our results  extend the findings of Kabluchko, Temesvari,and Th\"ale. Key techniques used in the proofs include the Blaschke--Petkantschin formula, Kubota's formula, and projections of beta distributed random vectors.
\end{abstract}

\maketitle
	
	\maketitle

\section{Introduction}

Let $X_1, \ldots, X_n \in \R^d$ be independent random vectors. Suppose that $2 \leqslant n \leqslant d+1$ and for some $k = 0, \ldots, n$ we have $X_1, \ldots, X_k$ uniformly distributed in the unit ball and $X_{k+1}, \ldots, X_n$ on the unit sphere. Then, their convex hull is almost surely an $(n-1)$-dimensional simplex. The classical result of Miles \cite{rM71} calculates the average volume of this simplex:
\begin{align*}
		    \E \textrm{Vol}_{n-1} (\textrm{conv}(X_1, \ldots, X_n)) &= \frac{1}{(n-1)!}\left(\frac{d}{d+1}\right)^k\frac{\Gamma\left(\frac{1}{2}n(d-1) + k + 1\right)}{\Gamma\left(\frac{1}{2}nd+k\right)} \\
            &\times \left(\frac{\Gamma(\frac{1}{2})d}{\Gamma(\frac{1}{2}(d+1))}\right)^{n-1}\prod_{j = 1}^{n-2}\frac{\Gamma\left(\frac{1}{2}(d-n+j+2)\right)}{\Gamma\left(\frac{1}{2}(d-n+j+1)\right)}.
\end{align*}

In fact, Miles calculated all positive integer moments of the volume;
see [\cite{rM71}, Theorem 2]. 

Later, Ruben and Miles \cite{RM80} generalized this result to the class of beta distributions in $\R^d$ with the probability densities defined as 
\[
	f_{d,\beta}(x) = c_{d, \beta}\cdot \left( 1 - |x|^2 \right)^{\beta}, \; |x| \leqslant 1,
\]
where $\beta > -1$ and normalizing constant $c_{d, \beta}$ is given by 
\[
	c_{d, \beta} = \frac{\Gamma(\frac{d}{2} + \beta + 1)}{\pi^{\frac d2}\Gamma(\beta + 1)}. 
\]

The uniform distribution in the unit ball corresponds to the case $\beta = 0$, while the uniform distribution on the unit sphere is obtained by taking a weak limit as $\beta$ approaches $-1$. Assuming that $X_i, i = 1, \ldots, n$, are  distributed with respect to $f_{d,\beta_i}$, they showed that for $k \in \mathbb{Z}_+, $

\begin{align}\label{miles}
    \E \left(\textrm{Vol}_{n-1}\left(\textrm{conv}(X_1, \ldots, X_n)\right)\right)^k &= ((n-1)!)^{-k}\frac{\Gamma\left(\frac{n(n-1+k)}{2} + \sum_{i = 1}^{n} \beta_i + 1\right)}{\Gamma\left(\frac{(n-1)(n+k)}{2} + \sum_{i = 1}^{n} \beta_i + 1\right)} \\ \notag
    &\times \prod_{j = 1}^{n-1} \frac{\Gamma\left(\frac{j+k}{2}\right)}{\Gamma\left(\frac{j}{2}\right)} \cdot \prod_{i = 1}^{n}\frac{\Gamma\left(\frac{n-1}{2} + \beta_i + 1\right)}{\Gamma\left(\frac{n - 1 + k}{2} + \beta_i + 1\right)}. 
\end{align}

Moreover, a similar result for the beta' polytopes has been obtained.

If we drop the condition $n \leqslant d+1,$ then the convex hull $\textrm{conv}(X_1, \ldots, X_n)$ is no longer a simplex, but a general convex polytope. In this case, Kabluchko et al. \cite{KTT19} obtained a formula for the average volume under the assumption that $\beta_1 = \ldots = \beta_n = \beta:$

\begin{align}\label{ktt}
    \E\mathrm{Vol}_d\left( \textrm{conv}(X_1, \ldots, X_n)\right) &= \frac{(d+1)\kappa_d}{\pi^{\frac{d+1}{2}}2^d} \cdot \binom{n}{d+1} \left(\frac{(d+1)}{2} +  \beta \right)\left(\frac{\Gamma(\frac{d+2}{2} + \beta)}{\Gamma(\frac{d+3}{2} + \beta)}\right)^{d+1} \\ \notag
    &\times \int_{-1}^1  (1-h^2)^{(d+1)\beta) + \frac{d^2 + 2d - 1}{2}}\cdot \left(F_{\beta + \frac{d-1}{2}}(h)\right)^{n - d - 1}  \dd h, 
\end{align}
where 
\[
	F_{\beta}(h) = c_{1, \beta}\int_{-1}^h (1 - x^2)^{\beta} \dd x, \;\; h \in [-1, 1] 
\]
is the CDF of the probability distribution associated with the density $f_{1, \beta}$.

They also considered the symmetrized random convex hulls and those
with the origin included, along with the case of the beta' random polytopes; see [\cite{KTT19}, Theorems 2.1, 2.2].
	
\section{Main results}
In this short note, we aim to generalize \eqref{ktt} to the case of non-identically distributed vectors.

Fix $\widebar{\beta} = (\beta_1, \ldots, \beta_n), \, \text{where} \, \beta_i > -1$. Let $X_1, \ldots, X_n$ be independent points in $\R^d$, such that $X_i$ is distributed according to the density $f_{d, \beta_i}$. 
	
Define the following random polytope
$\mathcal{P}_{n,d}^{\widebar{\beta}} = \textrm{conv}(X_1, \ldots, X_n)$. Our main result computes the expected volume of $\mathcal{P}_{n,d}^{\widebar{\beta}}$.

\begin{thm}
    We have 
\begin{align*}
    \E \Vol_d\left(\Pol{n}{d}{\beta}\right) &= \frac{\kappa_d}{\pi^{\frac{d+1}{2}}2^d} \cdot \sum_{i_1 < \ldots < i_{d+1}} \left(\frac{(d+1)^2}{2} +  \sum_{k = 1}^{d+1} \beta_{i_k} \right) \prod_{k = 1}^{d+1}\frac{\Gamma(\frac{d+2}{2} + \beta_{i_k})}{\Gamma(\frac{d+3}{2} + \beta_{i_k})} \\
    &\times \int_{-1}^1  (1-h^2)^{\sum_{k=1}^{d+1} \beta_{i_k} + \frac{d^2 + 2d - 1}{2}}\cdot \prod_{{k \neq i_1, \ldots, i_{d+1}}} F_{\beta_k + \frac{d-1}{2}}(h)  \dd h. 
\end{align*}
    
\end{thm}

We also compute expected value of the following functional of $\mathcal{P}_{n,d}^{\widebar{\beta}}$ introduced by Wieacker \cite{jW78}:

\[
	T_{a,b}^{k,d}(P) = \sum\limits_{F \in \mathcal{F}_{k}(P)}\textrm{dist}^a(F,0)\textrm{Vol}_k^b(F). 
\]

\begin{thm}
    We have 
\begin{align*}
    \E T_{a,b}^{d-1,d}&(\mathcal{P}_{n,d}^{\widebar{\beta}}) = d! \kappa_d \sum_{i_1 < \ldots < i_d} \frac{\prod c_{d, \beta_i}}{\prod c_{d-1, \beta_i}} \cdot \E \left(\textrm{Vol}_{d-1}\left(\textrm{conv}(X_{i_1}, \ldots, X_{i_d})\right)\right)^{b+1} \\ &\times \int_{-1}^1 |h|^a (1-h^2)^{\sum_{k=1}^d \beta_{i_k} + \frac{d-1}{2}(d + b + 1)}\cdot \prod_{{k \neq i_1, \ldots, i_d}} F_{\beta_k + \frac{d-1}{2}}(h)  \dd h,
\end{align*}
where the formula for $\E \left(\textrm{Vol}_{d-1}\left(\textrm{conv}(X_{i_1}, \ldots, X_{i_d})\right)\right)^{b+1}$ was given in \eqref{miles}. 
\end{thm}
Note that for $\beta_1 = \ldots = \beta_n$ this result was obtained in [\cite{KTT19}, Theorem 2.11].

In the next section, we collect the necessary facts about the beta distributions, which will be used in the following two sections containing the proofs of the theorems.

\section{Basic facts on beta-distributions.}
	
Fix $k \leqslant d$ and let $L$ be a $k$-dimensional linear  subspace $L = \{x \in \R^d: x_{k+1} = \ldots = x_d = 0\}$. Let $X$ be a random vector with  density $f_{d, \beta}$.

\begin{st} \label{beta1}
    Let $\pi_l : \R^d \to L$ be an orthogonal projection onto L, then $\pi_L(X)$ has density $f_{k, \beta + \frac{d-k}{2}}$. 
\end{st}

Consider the hyperplane $E_0 = \{x \in \R^d: x_d = 0\}$ and denote by $E_h = \{x \in \R^d:  x_d = h\}$ an affine hyperplane parallel to $E_0$.  Consider the following half-spaces:
\[
	E_h^+ = \{x \in \R^d: x_d >h \} \,\text{ and }\, E_h^- =  \{x \in \R^d:  x_d < h\}. 
\]
	
\begin{st} \label{beta2} 
We have
	\begin{gather*} 
		\PP \left( X \in E_h^+ \right) = 1 - F_{1,\beta + \frac{d-1}{2}}(h), \; h \in [-1,1]\\
		\PP\left( X \in E_h^- \right) =  F_{1,\beta + \frac{d-1}{2}}(h), \; h \in [-1,1]\\
		\PP \left( X \in E_h^-  \cap E_{-h}^+ \right)=  F_{1,\beta + \frac{d-1}{2}}(h) - F_{1,\beta + \frac{d-1}{2}}(-h),\;  h \in [0,1]. 
	\end{gather*}
\end{st}

Proofs of Statements \ref{beta1} and \ref{beta2} can be found in \cite{KTT19}. 

For $\widebar{\gamma} = (\gamma_1, \ldots, \gamma_{m+1})$, where $\gamma_i > -1$,  denote by $\Delta_m^{\widebar{\gamma}}=\textrm{Vol}_m\left(\mathcal{P}_{m+1,m}^{\widebar{\gamma}}\right)$ volume of random beta-simplex. 

\begin{lemma}\label{beta3}
	Let $E \in A(d, d-1)$ be an affine hyperplane, such that $\textrm{dist}(E, 0) = h$. Then
 \begin{align}
     \int_{E^d} \left(\textrm{Vol}_{d-1}\left(\textrm{conv}(x_1, \ldots, x_d)\right)\right)^k &\cdot \prod\limits_{i = 1}^d f_{d, \beta_i}(x_i) \prod \lambda_E(\dd x_i) \\ \notag
     &= \frac{\prod c_{d, \beta_i}}{\prod c_{d-1, \beta_i}} \cdot (1 - h^2)^{\sum_{i = 1}^d \beta_i + \frac{d-1}{2}(k+d)}\cdot \E(\Delta_{d-1}^{\widebar{\beta}})^k.
 \end{align}
\end{lemma}
\begin{proof}
    Due to rotational symmetry we can take $E = \{x \in \R^d: x_d = h\}$. Let L be the linear hyperplane parallel to $E$. If we denote by $x^* = \pi_l(x)$, we have $|x|^2 = |x^*|^2 + h^2$. Therefore 

    \begin{align*}
        &\int_{E^d} \left(\textrm{Vol}_{d-1}\left(\textrm{conv}(x_1, \ldots, x_d)\right)\right)^k \cdot \prod\limits_{i = 1}^d f_{d, \beta_i}(x_i) \prod \lambda_E(\dd x_i)  \\
       & = \prod c_{d, \beta_i} \int\limits_{(E \cap \B^d)^d}\left(\textrm{Vol}_{d-1}\left(\textrm{conv}(x_1, \ldots, x_d)\right)\right)^k\cdot \prod\limits_{i = 1}^d (1 - |x_i|^2)^{\beta_i}\prod \lambda_E(\dd x_i) \\ 
        &= \prod c_{d, \beta_i} \int\limits_{(L \cap \B^d\sqrt{1 - h^2})^d}\left(\textrm{Vol}_{d-1}\left(\textrm{conv}(x_1^*, \ldots, x_d^*)\right)\right)^k\cdot \prod\limits_{i = 1}^d (1 - |x^*_i|^2 - h^2)^{\beta_i}\prod \lambda_L(\dd x^*_i) \\
        &= \prod c_{d, \beta_i}\cdot  (1 - h^2)^{\sum \beta_i} \int\limits_{(L \cap \B^d\sqrt{1 - h^2})^d}\left(\textrm{Vol}_{d-1}\left(\textrm{conv}(x_1^*, \ldots, x_d^*)\right)\right)^k\cdot \prod\limits_{i = 1}^d \left(1 - \frac{|x^*_i|^2}{1 - h^2}\right)^{\beta_i}\prod \lambda_L(\dd x^*_i).
    \end{align*}
    Changing variables with $y_i = \frac{x^*_i}{\sqrt{1 - h^2}}$ leads to 
    \begin{align*}
        &\int_{E^d} \left(\textrm{Vol}_{d-1}\left(\textrm{conv}(x_1, \ldots, x_d)\right)\right)^k \cdot \prod\limits_{i = 1}^d f_{d, \beta_i}(x_i) \prod \lambda_E(\dd x_i)  \\
        &= \prod c_{d, \beta_i} \cdot (1 - h^2)^{\sum \beta_i + (d-1)(k + d)} \int\limits_{(\B^{d-1})^d}\left(\textrm{Vol}_{d-1}\left(\textrm{conv}(y_1, \ldots, y_d)\right)\right)^k\cdot \prod\limits_{i = 1}^d \left(1 - |y_i|^2 \right)^{\beta_i}\prod \lambda_L(\dd y_i) \\
        &= \frac{\prod c_{d, \beta_i}}{\prod c_{d-1, \beta_i}} \cdot (1 - h^2)^{\sum \beta_i + (d-1)(k + d)} \int\limits_{(\B^{d-1})^d}\left(\textrm{Vol}_{d-1}\left(\textrm{conv}(y_1, \ldots, y_d)\right)\right)^k\cdot \prod\limits_{i = 1}^d f_{d-1, \beta_i}(y_i)\prod \lambda_L(\dd y_i) \\
        &= \frac{\prod c_{d, \beta_i}}{\prod c_{d-1, \beta_i}} (1 - h^2)^{\sum \beta_i + (d-1)(k + d)} \cdot \E(\Delta_{d-1}^{\widebar{\beta}})^k, 
    \end{align*}
and the lemma follows. 
\end{proof}


\section{Proof of Theorem 2}

Recall that $T_{a,b}^{k,d}(P) = \sum\limits_{F \in \mathcal{F}_{k}(P)}\textrm{dist}^a(F,0)\textrm{Vol}_k^b(F)$. 
Let us calculate $\E T_{a,b}^{d-1,d}(\mathcal{P}_{n,d}^{\widebar{\beta}})$. Since points $X_1, \ldots, X_n$ are almost surely in general position, facets of $\mathcal{P}_{n,d}^{\widebar{\beta}}$ are almost surely $(d-1)$-dimensional simplices, since they can't have more than  $d$ vertices. So we have 

\begin{align*}
    \E T_{a,b}^{d-1,d}(\mathcal{P}_{n,d}^{\widebar{\beta}}) &= \E \left( \sum_{i_1 < \ldots < i_d} \mathbbm{1}(\conv(X_{i_1}, \ldots, X_{i_d}) \in \mathcal{F}_{d-1}(\Pol{n}{d}{\beta}))\right) \\ 
    &\times \textrm{dist}^a(\conv(X_{i_1},\ldots, X_{i_d}), 0)  
    \cdot \textrm{Vol}_{d-1}^b(\conv(X_{i_1},\ldots, X_{i_d}))  \\
    &= \sum_{i_1 < \ldots < i_d} \int\limits_{(\R^d)^d} 
    \PP\left(\conv(X_{i_1}, \ldots, X_{i_d}) \in \mathcal{F}_{d-1}(\Pol{n}{d}{\beta}) \, \textbar\, X_{i_1} = x_1, \ldots, X_{i_d} = x_d\right) \\
    &\times \textrm{dist}^a(\conv(x_1,\ldots, x_d), 0)  
    \cdot \textrm{Vol}_{d-1}^b(\conv(x_1,\ldots, x_d))\prod f_{d,\beta_{i_k}}(x_k) \prod \lambda(\dd x_i). 
\end{align*}

Applying Blaschke-Petkantchin formula we get 

\begin{align*}
    \int\limits_{(\R^d)^d} 
    \PP&\left(\conv(X_{i_1}, \ldots, X_{i_d}) \in \mathcal{F}_{d-1}(\Pol{n}{d}{\beta}) \, \textbar\, X_{i_1} = x_1, \ldots, X_{i_d} = x_d\right) \\
    &\times \textrm{dist}^a(\conv(x_1,\ldots, x_d), 0)  
    \cdot \textrm{Vol}_{d-1}^b(\conv(x_1,\ldots, x_d)) \prod f_{d,\beta_{i_k}}(x_k)\prod \lambda(\dd x_i) \\
    &= (d-1)!\frac{\omega_d}{\omega_1}\int\limits_{A_{d,d-1}}\int\limits_{E^d}\PP\left(\conv(X_{i_1}, \ldots, X_{i_d}) \in \mathcal{F}_{d-1}(\Pol{n}{d}{\beta}) \, \textbar\, X_{i_1} = x_1, \ldots, X_{i_d} = x_d\right) \\
    &\times \textrm{dist}^a(\conv(x_1,\ldots, x_d), 0)\cdot \textrm{Vol}_{d-1}^{b+1}(\conv(x_1,\ldots, x_d)) \prod f_{d,\beta_{i_k}}(x_k)\prod \lambda_{E}(\dd x_i)\mu_{d,d-1}(\dd E). 
\end{align*}

The conditional probability in the expression above is the probability that all $X_k$ with $k \neq i_1, \ldots, i_d$ lie in the same halfspace from the $\textrm{aff}(x_1, \ldots, x_d) = E$. If $\textrm{dist}(\conv(x_1,\ldots, x_d), 0) = h$, then, due to rotational invariance  and Statement \ref{beta2} we get 

\begin{align*}
    &\PP\left(\conv(X_{i_1}, \ldots, X_{i_d}) \in \mathcal{F}_{d-1}(\Pol{n}{d}{\beta}) \, \textbar \, X_{i_1} = x_1, \ldots, X_{i_d} = x_d\right) \\  &=\PP\left(\bigcap_{k \neq i_1, \ldots, i_d} \{X_{k} \in E^+\}\right) + \PP\left(\bigcap_{k \neq i_1, \ldots, i_d} \{X_{k} \in E^-\}\right) \\
    &= \prod_{{k \neq i_1, \ldots, i_d}} \left(1 - F_{\beta_k + \frac{d-1}{2}}(h)\right) + \prod_{{k \neq i_1, \ldots, i_d}} F_{\beta_k + \frac{d-1}{2}}(h). 
\end{align*}

Hence (again, due to symmetry)

\begin{align*}
    &(d-1)!\frac{\omega_d}{\omega_1}\int\limits_{A_{d,d-1}}\int\limits_{E^d}\PP\left(\conv(X_{i_1}, \ldots, X_{i_d}) \in \mathcal{F}_{d-1}(\Pol{n}{d}{\beta}) \, \textbar\, X_{i_1} = x_1, \ldots, X_{i_d} = x_d\right) \\
    &\times \textrm{dist}^a(\conv(x_1,\ldots, x_d), 0)\cdot \textrm{Vol}_{d-1}^{b+1}(\conv(x_1,\ldots, x_d)) \cdot \prod f_{d,\beta_{i_k}}(x_k)\prod \lambda_{E}(\dd x_i)\mu_{d,d-1}(\dd E) \\
    & = (d-1)!\frac{\omega_d}{\omega_1} \int_{G_{d,d-1}}\int_{-1}^1\int_{(L + h)^d} \left[ \prod_{{k \neq i_1, \ldots, i_d}} \left(1 - F_{\beta_k + \frac{d-1}{2}}(h)\right) + \prod_{{k \neq i_1, \ldots, i_d}} F_{\beta_k + \frac{d-1}{2}}(h) \right]\\
    &\times |h|^a \cdot \textrm{Vol}_{d-1}^{b+1}(\conv(x_1,\ldots, x_d)) \cdot \prod f_{d,\beta_{i_k}}(x_k) \prod \lambda_{L+h}(\dd x_i)\dd h \nu_{d,d-1}(\dd L) \\ 
    &= d! \kappa_d \int_0^1 \left[ \prod_{{k \neq i_1, \ldots, i_d}} \left(1 - F_{\beta_k + \frac{d-1}{2}}(h)\right) + \prod_{{k \neq i_1, \ldots, i_d}} F_{\beta_k + \frac{d-1}{2}}(h) \right] \cdot h^a \\
    &\times \int_{E^d} \textrm{Vol}_{d-1}^{b+1}(\conv(x_1,\ldots, x_d)) \cdot \prod f_{d,\beta_{i_k}}(x_k) \prod \lambda_{E}(\dd x_i) \dd h, 
\end{align*}

where $E = E_h$. 

Using Lemma \ref{beta3} we obtain 

\begin{align*}
    &d! \kappa_d \int_0^1 \left[ \prod_{{k \neq i_1, \ldots, i_d}} \left(1 - F_{\beta_k + \frac{d-1}{2}}(h)\right) + \prod_{{k \neq i_1, \ldots, i_d}} F_{\beta_k + \frac{d-1}{2}}(h) \right] \cdot h^a \\
    &\times \int_{E^d} \textrm{Vol}_{d-1}^{b+1}(\conv(x_1,\ldots, x_d)) \cdot \prod f_{d,\beta_{i_k}}(x_k) \prod \lambda_{E}(\dd x_i) \dd h \\
    & = d!\kappa_d \cdot \frac{\prod c_{d, \beta_i}}{\prod c_{d-1, \beta_i}} \cdot \E \left(\Delta_{d-1}^{\widebar{\beta_I}}\right)^{b+1} \cdot \int_0^1 h^a (1-h^2)^{\sum_{k=1}^d \beta_{i_k} + \frac{d-1}{2}(d + b + 1)} \\
    &\times \left[ \prod_{{k \neq i_1, \ldots, i_d}} \left(1 - F_{\beta_k + \frac{d-1}{2}}(h)\right) + \prod_{{k \neq i_1, \ldots, i_d}} F_{\beta_k + \frac{d-1}{2}}(h) \right] \dd h \\
    & = d!\kappa_d \cdot \frac{\prod c_{d, \beta_i}}{\prod c_{d-1, \beta_i}} \cdot \E \left(\Delta_{d-1}^{\widebar{\beta_I}}\right)^{b+1} \cdot \int_{-1}^1 |h|^a (1-h^2)^{\sum_{k=1}^d \beta_{i_k} + \frac{d-1}{2}(d + b + 1)} \\
    &\times \prod_{{k \neq i_1, \ldots, i_d}} F_{\beta_k + \frac{d-1}{2}}(h)  \dd h,
\end{align*}

where $\beta_I = (\beta_{i_1}, \ldots, \beta_{i_d})$. 

Hence, 
\begin{align*}
    \E T_{a,b}^{d-1,d}&(\mathcal{P}_{n,d}^{\widebar{\beta}}) = d! \kappa_d \sum_{i_1 < \ldots < i_d} \frac{\prod c_{d, \beta_i}}{\prod c_{d-1, \beta_i}} \cdot \E \left(\Delta_{d-1}^{\widebar{\beta_I}}\right)^{b+1} \\ &\times \int_{-1}^1 |h|^a (1-h^2)^{\sum_{k=1}^d \beta_{i_k} + \frac{d-1}{2}(d + b + 1)}\cdot \prod_{{k \neq i_1, \ldots, i_d}} F_{\beta_k + \frac{d-1}{2}}(h)  \dd h. 
\end{align*}

\section{Proof of Theorem 1}

First let us consider only the case when all $\beta_i > -\frac{1}{2}$. Then if $L$ is an arbitrary hyperplane in $G(d+1, d)$, by Statement \ref{beta1} we have 
\[
    \Pol{n}{d+1}{\beta - \frac{1}{2}}\textbar_L = \Pol{n}{d}{\beta}.  
\]

Applying Kubota formula, we get

\begin{align*}
    V_d(\Pol{n}{d+1}{\beta - \frac{1}{2}}) = (d+1)\frac{\kappa_{d+1}}{2\kappa_{d}}\int_{G(d+1, d)} \textrm{Vol}_d\left(\Pol{n}{d+1}{\beta - \frac{1}{2}}\textbar_L\right) \nu_{d+1, d} (\dd L),
\end{align*}

therefore 

\begin{align*}
    \E V_d\left((\Pol{n}{d+1}{\beta - \frac{1}{2}}\right) &= (d+1)\frac{\kappa_{d+1}}{2\kappa_{d}}\int_{G(d+1, d)} \E\textrm{Vol}_d\left(\Pol{n}{d+1}{\beta - \frac{1}{2}}\textbar_L\right) \nu_{d+1, d} (\dd L) \\
    & = (d+1)\frac{\kappa_{d+1}}{2\kappa_{d}} \int_{G(d+1, d)} \E\textrm{Vol}_d\left(\Pol{n}{d}{\beta}\right) \nu_{d+1, d} (\dd L)  = (d+1)\frac{\kappa_{d+1}}{2\kappa_{d}} \E \textrm{Vol}_d \Pol{n}{d}{\beta}. 
\end{align*}

Since $V_d\left((\Pol{n}{d+1}{\beta - \frac{1}{2}}\right)  = \frac{1}{2} T_{0,1}^{d,d+1}(\Pol{n}{d+1}{\beta - \frac{1}{2}})$, we obtain 

\begin{align*}
    &\E \Vol_d\left(\Pol{n}{d}{\beta}\right) = \frac{\kappa_d}{(d+1)\kappa_{d+1}}\E T_{0,1}^{d,d+1}(\Pol{n}{d+1}{\beta - \frac{1}{2}}) = \frac{\kappa_d}{(d+1)\kappa_{d+1}}(d+1)!\kappa_{d+1}\\
    &\times \sum_{i_1 < \ldots < i_{d+1}}
    \frac{\prod c_{{d+1}, \beta_i - \frac{1}{2}}}{\prod c_{d, \beta_i - \frac{1}{2}}} \cdot \E \left(\Delta_{d}^{\widebar{\beta_I - \frac{1}{2}}}\right)^{2} \cdot \int_{-1}^1  (1-h^2)^{\sum_{k=1}^{d+1} \beta_{i_k} - \frac{d+1}{2} + \frac{d}{2}(d + 3)}\cdot \prod_{{k \neq i_1, \ldots, i_{d+1}}} F_{\beta_k + \frac{d-1}{2}}(h)  \dd h. 
\end{align*}

Applying \eqref{miles}, we obtain

\begin{align*}
    \E \left(\Delta_{d}^{\widebar{\beta_I - \frac{1}{2}}}\right)^{2} &= (d!)^{-2}\frac{\Gamma\left(\frac{(d+1)(d+2)}{2} + \sum_{k = 1}^{d+1} (\beta_{i_k} - \frac{1}{2}) + 1\right)}{\Gamma\left(\frac{d(d+3)}{2} + \sum_{k = 1}^{d+1} (\beta_{i_k} - \frac{1}{2}) + 1\right)} \cdot \prod_{j = 1}^d \frac{\Gamma\left(\frac{j+2}{2}\right)}{\Gamma\left(\frac{j}{2}\right)} \cdot \prod_{k = 1}^{d+1}\frac{\Gamma\left(\frac{d}{2} + \beta_{i_k} - \frac{1}{2} + 1\right)}{\Gamma\left(\frac{d + 2}{2} + \beta_{i_k} - \frac{1}{2} + 1\right)} \\
    &= \frac{1}{d!2^d}\cdot\left(\frac{(d+1)^2}{2} +  \sum_{k = 1}^{d+1} \beta_{i_k} \right) \cdot \prod_{k = 1}^{d+1}\frac{1}{\frac{d+1}{2} + \beta_{i_k}}, 
\end{align*}

hence 

\begin{align*}
    \E \Vol_d\left(\Pol{n}{d}{\beta}\right) &= \frac{\kappa_d}{2^d}
    \cdot \sum_{i_1 < \ldots < i_{d+1}}
    \prod_{k = 1}^{d+1}\frac{c_{d+1, \beta_{i_k} - \frac{1}{2}}}{c_{d, \beta_{i_k} - \frac{1}{2}}} \cdot \left(\frac{(d+1)^2}{2} +  \sum_{k = 1}^{d+1} \beta_{i_k} \right)\cdot\prod_{k = 1}^{d+1}\frac{1}{\frac{d+1}{2} + \beta_{i_k}} \\
    &\times \int_{-1}^1  (1-h^2)^{\sum_{k=1}^{d+1} \beta_{i_k} + \frac{d^2 + 2d - 1}{2}}\cdot \prod_{{k \neq i_1, \ldots, i_{d+1}}} F_{\beta_k + \frac{d-1}{2}}(h)  \dd h  \\
    &= \frac{\kappa_d}{2^d} \cdot \sum_{i_1 < \ldots < i_{d+1}} \left(\frac{(d+1)^2}{2} +  \sum_{k = 1}^{d+1} \beta_{i_k} \right) \prod_{k = 1}^{d+1}\frac{\Gamma(\frac{d+2}{2} + \beta_{i_k})}{\sqrt{\pi}\Gamma(\frac{d+1}{2} + \beta_{i_k})} \cdot\prod_{k = 1}^{d+1}\frac{1}{\frac{d+1}{2} + \beta_{i_k}} \\
    &\times \int_{-1}^1  (1-h^2)^{\sum_{k=1}^{d+1} \beta_{i_k} + \frac{d^2 + 2d - 1}{2}}\cdot \prod_{{k \neq i_1, \ldots, i_{d+1}}} F_{\beta_k + \frac{d-1}{2}}(h)  \dd h \\
    &= \frac{\kappa_d}{\pi^{\frac{d+1}{2}}2^d} \cdot \sum_{i_1 < \ldots < i_{d+1}} \left(\frac{(d+1)^2}{2} +  \sum_{k = 1}^{d+1} \beta_{i_k} \right) \prod_{k = 1}^{d+1}\frac{\Gamma(\frac{d+2}{2} + \beta_{i_k})}{\Gamma(\frac{d+3}{2} + \beta_{i_k})} \\
    &\times \int_{-1}^1  (1-h^2)^{\sum_{k=1}^{d+1} \beta_{i_k} + \frac{d^2 + 2d - 1}{2}}\cdot \prod_{{k \neq i_1, \ldots, i_{d+1}}} F_{\beta_k + \frac{d-1}{2}}(h)  \dd h. 
\end{align*}

So far we have proved Theorem 1 only for the case when $\beta_i > -\frac{1}{2}.$ Applying the argument of analytic continuation finishes the proof. 

	\bibliographystyle{plain}
	\bibliography{bib}

\begin{thebibliography}{1}

\bibitem{KTT19}
Z.~Kabluchko, D.~Temesvari, and Ch. Th\"{a}le.
\newblock Expected intrinsic volumes and facet numbers of random beta-polytopes.
\newblock {\em Math. Nachr.}, 292(1):79--105, 2019.

\bibitem{rM71}
R.~Miles.
\newblock Isotropic random simplices.
\newblock {\em Adv. in Appl. Probab.}, 3:353--382, 1971.

\bibitem{RM80}
Harold Ruben and R.~E. Miles.
\newblock A canonical decomposition of the probability measure of sets of isotropic random points in {{\(\mathbb{R}^n\)}}.
\newblock {\em J. Multivariate Anal.}, 10:1--18, 1980.

\bibitem{jW78}
JA~Wieacker.
\newblock Einige probleme der polyedrischen approximation.
\newblock {\em Freiburg im Breisgau: Diplomarbeit}, 1978.

\end{thebibliography}
	
\end{document}